\documentclass[a4paper,10pt]{amsart}
\usepackage[utf8]{inputenc}
\usepackage[pdftex]{graphicx}
\usepackage{a4wide}
\usepackage[english]{babel}
\usepackage[T1]{fontenc}
\usepackage{amsmath,amssymb,amsthm,stmaryrd}
\usepackage{mathrsfs}
\usepackage{esint}

\newcommand{\R}{\mathbb{R}}

\newcommand{\N}{\mathbb{N}}
\newcommand{\Z}{\mathbb{Z}}

\newcommand{\Nc}{\mathcal{N}}
\newcommand{\Pro}{\mathbb{P}}
\newcommand{\Qt}{\widetilde{Q}}
\newcommand{\Qb}{\bar{Q}}

\newcommand{\A}{\mathscr{A}}

\newcommand{\epsb}{\partial_\varepsilon}

\newcommand{\abs}[1]{|#1|}
\newcommand{\Norm}[2]{\|#1\|_{#2}}
\newcommand{\BNorm}[2]{\Big\|#1\Big\|_{#2}}

\newtheorem{theorem}[equation]{Theorem}
\newtheorem{lemma}[equation]{Lemma}
\newtheorem{openproblem}[equation]{Open problem}
\newtheorem{corollary}[equation]{Corollary}

\theoremstyle{definition}
\newtheorem{defin}[equation]{Definition}
\newtheorem*{defin*}{Definition}

\newtheorem{remark}[equation]{Remark}
\newtheorem*{remark*}{Remark}
\newcommand{\rref}[1]{$\left(\ref{#1}\right)$}
\numberwithin{equation}{section}

\title{Almost Lipschitz-continuous wavelets in metric spaces via a new randomization of dyadic cubes}

\author{Tuomas Hyt\"onen \and Olli Tapiola}

\address{Department of Mathematics and Statistics, P.O.B. 68 (Gustaf H\"allstr\"omin katu 2b), FI-00014 University of Helsinki, Finland}
\email{tuomas.hytonen@helsinki.fi}
\email{olli.tapiola@helsinki.fi}

\date{October 8, 2013}
\keywords{metric space, dyadic cube, doubling property, wavelets, spline functions}
\subjclass[2010]{30L99 (Primary); 42C40, 41A15, 60D05 (Secondary)}
\thanks{Both authors are supported by the European Union through the ERC Starting Grant "Analytic-probabilistic methods for borderline singular integrals". T. H. is also supported by the Academy of Finland, grants 130166 and 133264.}

\begin{document}

\maketitle

\begin{abstract}
In any quasi-metric space of homogeneous type, Auscher and Hyt\"onen recently gave a construction of orthonormal wavelets with H\"older-continuity exponent $\eta>0$. However, even in a metric space, their exponent is in general quite small. In this paper, we show that the H\"older-exponent can be taken arbitrarily close to $1$ in a metric space. We do so by revisiting and improving the underlying construction of random dyadic cubes, which also has other applications.
\end{abstract}

\section{Introduction}

The representation of functions in terms special orthogonal expansions, like wavelet bases, is one of the central themes and tools in harmonic analysis, approximation theory, and their applications. While many kinds of useful wavelets have long been known in the Euclidean space $\R^d$ and some other special geometries, the first comprehensive construction of regular orthonormal wavelet bases (and not just frames, which had been known before \cite{denghan}) in abstract metric or even quasi-metric spaces was only recently obtained by Auscher and one of us \cite{auscherhytonen}. More precisely, it was shown in \cite{auscherhytonen} that any space of homogeneous type $(X,d,\mu)$ in the sense of Coifman and Weiss (i.e., a quasi-metric space equipped with a doubling measure) supports an orthonormal basis $\psi^k_\alpha$ with localization and regularity properties of the form
\begin{equation*}
\begin{split}
  \abs{\psi^k_\alpha(x)} &\leq\frac{\exp(-\gamma d(x,y^k_\alpha)\delta^{-k})}{\sqrt{\mu (B(y^k_\alpha,\delta^k))}}, \\
  \abs{\psi^k_\alpha(x)-\psi^k_\alpha(y)} &\leq\frac{\exp(-\gamma d(x,y^k_\alpha)\delta^{-k})}{\sqrt{\mu(B(y^k_\alpha,\delta^k))}}\Big(\frac{d(x,y)}{\delta^k}\Big)^\eta,\qquad
  d(x,y)\leq\delta^k.
\end{split}
\end{equation*}
That is, $\psi^k_\alpha$ is localized in a ball $B(y^k_\alpha,\delta^k)$, up to an exponentially decaying tail, and satisfies a H\"older-continuity estimate of exponent $\eta$ on the scale $\delta^k$.

The main theme of this paper is the value of the H\"older-exponent $\eta$ above. The construction of \cite{auscherhytonen} shows that it is some small but strictly positive number, $\eta>0$, and this is essentially the best that one can hope for in a general space of homogeneous type; indeed, it is known that there may only exist non-trivial H\"older-continuous functions for H\"older-exponents below a small threshold.

However, if we restrict ourselves to the most important case of actual metric spaces, the situation changes drastically. Now the distance itself is Lipschitz-continuous in both variables, and an abundance of other Lipschitz-continuous functions may be easily derived from it. Thus it is perfectly reasonable to inquire about the existence of Lipschitz-continuous wavelets (i.e., $\eta=1$ above), but the construction offered in \cite{auscherhytonen} does not seem to capture any substantial benefit from the restriction to an actual metric.

In this paper, we address the problem by offering a different construction of the metric wavelets, which allows us to obtain H\"older-regularity of any exponent $\eta<1$, strictly below but arbitrarily close to one. More precisely, we offer a new construction of the random dyadic cubes that served as the working engine of \cite{auscherhytonen}. In fact, the H\"older-regularity of the wavelets is a direct reflection of the probabilistic boundary regularity of the random dyadic cubes: given a point $x\in X$ , its probability of ending up close to the boundary of a random cube $Q^k_\alpha$ of side-length $\delta^k$ satisfies
\begin{equation}\label{eq:cubes}
  \mathbb{P}\Big(x\in Q^k_\alpha,\ d(x,(Q^k_\alpha)^c)<\varepsilon\Big)\leq C\Big(\frac{\varepsilon}{\delta^k}\Big)^\eta.
\end{equation}

The notion of random dyadic cubes has been instrumental for several recent advances in harmonic analysis, both in Euclidean \cite{hytonensharpbound, nazarovtreilvolberg, laceysawyer_etal, tolsa} and more abstract spaces \cite{nazarovreznikovvolberg}; however, aside from the construction of wavelets, the value of the boundary exponent $\eta$ in \eqref{eq:cubes} seems to be inessential for most of these applications, as long as it is positive. In metric spaces, the construction of these random cubes was first given by Martikainen and one of us \cite{hytonenmartikainen}, and then simplified and elaborated in \cite{hytonenkairema} and further in \cite{auscherhytonen}. In all these papers, the starting point of the randomization was the well-known deterministic construction due to Christ \cite{christ}. The abstract cubes of Christ are determined by two objects: the centre-points $z^k_\alpha$, and a partial order $\leq$, the ``parent-child'' relation, which determines the inclusion properties between cubes of 
different generations. When randomizing his construction in \cite{auscherhytonen,hytonenkairema,
hytonenmartikainen}, it 
hence appeared natural to randomize both the choice of the centre-points and that of the parent-child relation. Our present approach differs from these in that we keep the centre-points fixed, and only randomize the parent-child relation, which is actually simpler than the earlier abstract constructions. The reason that this strategy was not discovered before is probably its deviation from the Euclidean intuition: in $\R^d$, the cubes of a given generation are determined by their centre-points alone, and the construction of random cubes amounts to randomly shifting these centres. In an abstract space, it was a natural first guess that we do at least the same, and then whatever additional corrections are necessary. Even in such problems, where the exact value of the boundary exponent $\eta$ is irrelevant, we believe that our new approach may be useful for its simplicity.

We conclude the introduction by shortly discussing some motivation for our problem. Orthonormal wavelets serve as basic building blocks for resolutions of the identity in the style of the Littlewood--Paley theory, or the Calder\'on reproducing formula. Studying the rates of convergence of such resolutions for different functions naturally leads to function space norms of Besov or Triebel--Lizorkin type:
\begin{equation*}
  \Big(\sum_k [\delta^{-ks}\Norm{\Psi^k f}{L^p(\mu)}]^q\Big)^{1/q}\quad\text{or}\quad
  \BNorm{\Big(\sum_k [\delta^{-ks}\abs{\Psi^k f}]^q\Big)^{1/q}}{L^p(\mu)},\qquad
  \Psi^k f:=\sum_\alpha \psi^k_\alpha\langle\psi^k_\alpha,f\rangle,
\end{equation*}
where $s\in\R$ is a smoothness index. However, generally speaking such norms are well-behaved (for instance, equivalent under different choices of the wavelet resolution) only if the smoothness index is bounded by the regularity of the wavelets, $\abs{s}<\eta$. Our new construction, which provides wavelets of arbitrary H\"older-regularity $\eta<1$, should open the way for the wavelet approach to the metric space theory of function spaces of any smoothness index $s\in(-1,1)$. Pursuing this line of research in detail is, however, left for future investigation. For earlier developments in this direction, but based on other resolutions of the identity than ones arising from wavelets, see e.g. the extensive theory of Besov and Triebel--Lizorkin spaces built by Han, M\"uller and Yang \cite{hanmulleryang}, and the work of Yang and Zhou \cite{yangzhou} on a classical problem of Coifman and Weiss \cite{coifmanweiss} on the characterization of Hardy spaces on spaces of homogeneous type. However, the results of \cite{
hanmulleryang,yangzhou} are set up in somewhat more restrictive ``reverse doubling'' spaces, a condition that was shown to be unnecessary for the metric wavelet theory of~\cite{auscherhytonen}, and likewise for its present elaboration.

The paper is organized as follows. Section~\ref{set-up_etc} collects the basic background from earlier related papers, Section~\ref{boundaries_and_iterations} introduces general tools for handling the boundary regions of sets that we want to estimate, and the two subsequent sections present our new randomization of the dyadic cubes. Indeed, we provide two different randomizations serving different purposes: the aspects common to both versions are treated in Section~\ref{random_dyadic_systems}, and the individual features of the two versions in Section~\ref{two_different}. Finally, the application to the construction of wavelets is presented in Section~\ref{application}.

\section{Set-up, dyadic points and cubes}
\label{set-up_etc}

Throughout this paper $(X,d)$ is a metric space that satisfies the following (geometrical) \emph{doubling property}: there exists a constant 
$M \in \N := \{0,1,2,\ldots\}$ such that for every $x \in X$ and every $r > 0$, the open ball $B(x,r) := \{y \in X : d(x,z) < r\}$ 
can be covered by at most $M$ open balls of radius $r/2$. We call the space $(X,d)$ a \emph{doubling metric space} and the constant $M$ the \emph{doubling constant} 
of $X$. We need the following well-known properties of doubling metric spaces repeatedly in several proofs.

\begin{lemma}
  \label{dms_properties}
  For any doubling metric space $(X,d)$ with a doubling constant $M$ the following properties hold:
  \begin{enumerate}
   \item[$1)$] Any ball $B(x,r)$ can be covered by at most $\lfloor M\delta^{-\log_2 M} \rfloor$ balls $B(x_i, \delta r)$ for every 
               $\delta \in (0,1]$.
   \item[$2)$] Any ball $B(x,r)$ contains at most $\lfloor M\delta^{-\log_2 M} \rfloor$ centres $x_i$ of pairwise disjoint
               balls $B(x_i,\delta r)$ for every $\delta \in (0,1]$.
  \end{enumerate}
\end{lemma}

\begin{proof}
  See e.g. Lemma 2.3 of \cite{hytonenframework}.
\end{proof}

\begin{lemma}
  \label{maximal_subsets}
  In every doubling metric space $(X,d)$ for any $\delta > 0$ there exists a countable maximal set $\A_\delta \subseteq X$ of $\delta$-separated points:
     \begin{enumerate}
       \item[$\bullet$] $d(x,y) \ge \delta$ for every $x,y \in \A_\delta$, $x \neq y$
       \item[$\bullet$] $\underset{x \in \A_\delta}{\min} \ d(x,z) < \delta$ for every $z \in X$.
     \end{enumerate}
\end{lemma}

\begin{proof}
  Let $x_0 \in X$ and $r \ge \delta$. By Lemma \ref{dms_properties}, a finite maximal $\delta$-separated subset $\A_1 \subseteq B(x_0, r)$ exists. Also by Lemma 
  \ref{dms_properties}, for every $k \ge 2$ there exists a finite maximal $\delta$-separated subset $\A_k \subseteq B(x_0, kr) \setminus \bigcup_{z \in \A^{k-1}} B(z,\delta)$ 
  where $\A^{k-1} = \bigcup_{i=1}^{k-1} \A_i$. By construction, the set $\A^k$ is $\delta$-separated, finite and maximal in $B(x_0, kr)$ and thus, we can set 
  $\A_\delta = \bigcup_{i=1}^{\infty} \A_i$. The minimum in the second condition is attained by the first condition and the doubling property of the space $(X,d)$.
\end{proof}

\begin{remark}
  \label{maximal_sets_remark}
  \begin{enumerate}
   \item[1)] If we replace the minimum with infimum and allow the set $\A_\delta$ to be uncountable, the claim of Lemma \ref{maximal_subsets} holds for even non-doubling metric spaces. 
             We can prove this claim quite simply by applying Zorn's lemma to the collection of $\delta$-separated subsets of $X$.
  
  \item[2)] In particular, we can choose maximal sets of $\delta$-separated points from any subset of $X$. If the subset is bounded and the space is doubling, 
            the maximal set is finite.
  \end{enumerate}
\end{remark}

\subsection{Dyadic points}

The following theorem gives us sets of so called \emph{dyadic points} that resemble the centre-points of dyadic cubes in the Euclidean space.

\begin{theorem}
  \label{construction_of_dyadic_points}
  In every doubling metric space $(X,d)$ for any $\delta \in (0,1/2)$ there exist sets $\A_k := \{z_\alpha^k : \alpha \in \Nc_k \}$ for every 
  $k \in \Z$ such that
  \begin{gather}
    \A_k \subseteq \A_{k+1}, \label{inclusion_of_dyadic_points} \\
    d(z_\alpha^k, z_\beta^k) \ge \delta^k \text{ for } \alpha \neq \beta, \label{separation_of_dyadic_points} \\
    \underset{\alpha}{\min} \ d(x,z_\alpha^k) < \delta^k \text{ for every } x \in X, \label{maximality_of_dyadic_points}
  \end{gather}
  where $\Nc_k = \{0, 1, \ldots, n_k\}$, if the space $(X,d)$ is bounded, and $\Nc_k = \N$ otherwise.
\end{theorem}

If we weaken the property \rref{maximality_of_dyadic_points} to the form
\begin{eqnarray*}
  \min_\alpha d(x,z_\alpha^k) < 2\delta^k \text{ for every } x \in X,
\end{eqnarray*}
the proof is quite simple when we use of the first part of Lemma \ref{dms_properties} and induction (see e.g. Lemma 2.1 in \cite{auscherhytonen}). However, in the stronger form the proof becomes 
somewhat technical and will be postponed to Appendix A. By the proof and Remark \ref{maximal_sets_remark}, the claim of Theorem 
\ref{construction_of_dyadic_points} holds for any metric space if we allow the sets $\A_k$ to be uncountable and weaken the property \rref{maximality_of_dyadic_points} to the form
\begin{eqnarray*}
  \inf_\alpha d(x,z_\alpha^k) < \delta^k \text{ for every } x \in X.
\end{eqnarray*}

Next, let us formulate a lemma for the relation between the dyadic points or, more precisely, between the index pairs $(k,\alpha)$, $k \in \Z$, $\alpha \in \Nc_k$.
We formulate the lemma in such a way that the properties of the non-random dyadic cubes in Theorem \ref{open_and_closed_dcubes} hold also for the 
random cubes in later sections. For the non-random cubes, a bit simpler formulation would be sufficient (see Lemma 2.10 in \cite{hytonenkairema}).

\begin{lemma}[Partial order of dyadic points]
  \label{partial_order_of_dyadic_points}
  Let $(X,d)$ be a doubling metric space with a doubling constant $M$ and $\A_k := \{z_\alpha^k : \alpha \in \Nc_k \}$ be sets given by Theorem \ref{construction_of_dyadic_points} for $\delta \in (0,1/2)$, 
  $k \in \Z$. Let $r_k \in [(1/4)\delta^k, (1/2)\delta^k]$ for every $k \in \Z$. Then there exists a partial order $\le$ among the pairs 
  $(k,\alpha)$ such that
  \begin{enumerate}
    \item[$\bullet$] if $z_\beta^{k+1} \in B(z_\alpha^k,r_k)$, then $(k+1,\beta) \le (k,\alpha)$;
    \item[$\bullet$] if $(k+1,\beta) \le (k,\alpha)$, then $z_\beta^{k+1} \in B(z_\alpha^k, 4r_k)$;
    \item[$\bullet$] for every $(k+1,\beta)$, there is exactly one $(k,\alpha) \ge (k+1,\beta)$, called its \emph{parent};
    \item[$\bullet$] for every $(k,\alpha)$, there are between $1$ and $\lceil M^3 \delta^{-\log_2 M} \rceil$ pairs $(k+1,\beta) \le (k,\alpha)$,
                     called its \emph{children};
    \item[$\bullet$] $(l,\beta) \le (k,\alpha)$ if and only if $l \ge k$ and there exist $(j+1,\gamma_{j+1}) \le (j,\gamma_j)$ for 
                     every $j = k, k+1, \ldots, l-1$ and for some $\gamma_k = \alpha, \gamma_{k+1}, \ldots, \gamma_{l-1}, \gamma_l = \beta$;
                     then $(l,\beta)$ and $(k,\alpha)$ are called one another's \emph{descendant} and \emph{ancestor}, respectively.
  \end{enumerate}
\end{lemma}

\begin{proof}
  Since the sets $\mathscr{A}_k$, $k \in \Z$, are indexed by natural numbers, we can talk about the smallest index $\alpha$ of every 
  subset of $\mathscr{A}_k = \{z_\alpha^k : \alpha \in \Nc_k \}$. This is essential for the partial order we are defining.

  Given a pair $(k+1,\beta)$, check whether there exists $z_\alpha^k \in \mathscr{A}_k$ such that $z_\beta^{k+1} \in B\left(z_\alpha^k, 
  r_k\right)$. If one exists, we decree that $(k+1,\beta) \le (k,\alpha)$, since it is necessarily unique by \rref{separation_of_dyadic_points}. 
  If no such $z_\alpha^k$ exist, we will look at every $z_\gamma^k \in \mathscr{A}_k$ for which $z_\beta^{k+1} \in B\left(z_\gamma^k, 4r_k\right)$. 
  At least one such $z_\gamma^k$ exists by \rref{maximality_of_dyadic_points}. From these,
  we choose the one with the smallest index $\theta$, and decree that $(k+1,\beta) \le (k,\theta)$. In either case, we decree 
  that $(k+1,\beta)$ is not related to any other pair $(k,\nu)$. We also decree that $(k,\alpha) \le (k,\alpha)$ for every $k \in \Z$ 
  and $\alpha \in \Nc_k$ and finally extend $\le$ by transitivity to obtain a partial ordering.

  Let $z_\alpha^k \in \A_k$. Since $\mathscr{A}_k \subseteq \mathscr{A}_{k+1}$, we know that $z_\alpha^k = z_\beta^{k+1}$ for some $\beta \in \N$. Since 
  $z_\beta^{k+1} \in B(z_\alpha^k,r_k)$, we know that $(k+1,\beta) \le (k,\alpha)$ and thus, $(k,\alpha)$ has at least one child. 
  On the other hand, if $(k+1,\beta) \le (k,\alpha)$, then $d(z_\alpha^k, z_\beta^{k+1}) < 2\delta^k$ and $d(z_\beta^{k+1}, z_\gamma^{k+1}) 
  \ge \delta^{k+1}$ for any $z_\gamma^{k+1} 
  \neq z_\beta^{k+1}$. For these $z_\gamma^{k+1}$ and $z_\beta^{k+1}$ the balls $B(z_\gamma^{k+1},(1/2)\delta^{k+1})$ and 
  $B(z_\beta^{k+1},(1/2)\delta^{k+1})$ are disjoint so by Lemma \ref{dms_properties}, 
  there are at most $\lceil M ( 4 / \delta)^{\log_2 M} \rceil = \lceil M^3 \delta^{-\log_2 M} \rceil$ of centres of these balls in $B(z_\alpha^k, 2\delta^k)$.
\end{proof}

We call $r_k$ the \emph{inner radius of level $k$} and $R_k := 4r_k$ the \emph{outer radius of level $k$}. In a similar fashion, we call the ball $B(z_\alpha^k, r_k)$ the \emph{inner ball 
of $z_\alpha^k$} and $B(z_\alpha^k, R_k)$ the \emph{outer ball of $z_\alpha^k$}.

\subsection{Open and closed dyadic cubes}

With the help of Theorem \ref{construction_of_dyadic_points} and Lemma \ref{partial_order_of_dyadic_points}, we can now formulate 
the theorem for open and closed dyadic cubes.

\begin{theorem}
  \label{open_and_closed_dcubes}
  Let $(X,d)$ be a doubling metric space with a doubling constant $M$ and $\delta \in (0,1/60]$. Given sets of dyadic points $\mathscr{A}_k := \{z_\alpha^k : \alpha \in \N\}$ that satisfy properties 
  \rref{inclusion_of_dyadic_points}, \rref{separation_of_dyadic_points} and \rref{maximality_of_dyadic_points} for every $k \in \Z$, we can construct families of sets $\Qt_\alpha^k$ and $\Qb_\alpha^k$ (called \emph{open} and 
  \emph{closed dyadic cubes}) such that
  \begin{gather}
    \text{\emph{int}} \Qb_\alpha^k = \Qt_\alpha^k, \ \ \overline{\Qt_\alpha^k} = \Qb_\alpha^k; \label{interior_and_closure_of_dcubes} \\
    \Qb_\alpha^k \cap \Qt_\beta^k = \emptyset \ \text{ if } \alpha \neq \beta; \label{intersection_of_open_and_closed_dcubes} \\
    X = \bigcup_\alpha \Qb_\alpha^k \ \text{ for every } k \in \Z; \label{union_of_closed_dcubes} \\
    B(z_\alpha^k, \frac{1}{5} \delta^k) \subseteq \Qt_\alpha^k \subseteq \Qb_\alpha^k \subseteq B(z_\alpha^k, 3\delta^k); \label{inclusion_of_dcubes} \\
    \Qb_\alpha^k = \bigcup_{\beta: (l,\beta) \le (k,\alpha)} \Qb_\beta^l \ \ \text{ for every } l \ge k. \label{union_of_descendants_of_closed_dcubes}
  \end{gather}
\end{theorem}

The proof of Theorem \ref{open_and_closed_dcubes} is analoguous to the proof of Proposition 2.11 in \cite{hytonenkairema}. Since we wanted to formulate Lemma 
\ref{partial_order_of_dyadic_points} in a more general way than it was formulated in \cite{hytonenkairema}, our inclusion property \rref{inclusion_of_dcubes} 
is a bit weaker than in \cite{hytonenkairema}. This is due to Lemma 3.1 in \cite{hytonenkairema}: our formulation of Lemma \ref{partial_order_of_dyadic_points} 
does not give as sharp a result as its formulation in \cite{hytonenkairema} gives. However, for the results in this paper, this is insignificant.

In \cite{hytonenkairema} it is also shown that in every doubling (quasi)metric space we can construct \emph{half-open dyadic cubes} that resemble the 
standard half-open dyadic cubes in $\R^n$, but we do not need them in this paper.

\section{$\varepsilon$-boundaries of sets and approximations of cubes}
\label{boundaries_and_iterations}

In this section we introduce two new definitions and prove some results related to them. Some results are somewhat technical but we need them in the following sections.

\subsection{$\varepsilon$-boundaries of sets}

Since we are interested in the boundary regions of cubes, let us define what we mean with boundary regions of sets.

\begin{defin}
  The \emph{$\varepsilon$-boundary} of a set $A \subseteq X$ is
  \begin{eqnarray*}
    \epsb A := \{ x \in A : d(x,A^c) < \varepsilon\} \cup \{ x \in A^c : d(x,A) < \varepsilon \}.
  \end{eqnarray*}
\end{defin}

The following properties of $\varepsilon$-boundaries are straightforward consequences of the definition.

\begin{lemma}
  \label{epsilon-boundary_lemma}
  \begin{enumerate}
    \item[a)] For any set $A \subseteq X$ we have $\epsb A = \epsb (A^c)$.
    
    \item[b)] For every $\varepsilon, r > 0$ and $x \in X$ we have
               \begin{align}
                 \epsb B(x,r) \subseteq B(x,r + \varepsilon) \setminus \bar{B}(x,r - \varepsilon). \label{epsilon-boundary_ball}
               \end{align}

    \item[c)] The $\varepsilon$-boundary of a union of sets is a subset of the union of $\varepsilon$-boundaries of those sets:
               \begin{align}
                 \epsb \left( \bigcup_i A_i \right) \subseteq \bigcup_i \epsb A_i.
               \end{align}
               
    \item[d)] The $\varepsilon$-boundary of an intersection of sets is a subset of the union of $\varepsilon$-boundaries of those sets:
               \begin{align}
                 \partial_\varepsilon \left( \bigcap_i A_i \right) \subseteq \bigcup_i \partial_\varepsilon A_i.
               \end{align}
  \end{enumerate}
\end{lemma}

\subsection{Approximated dyadic cubes}

As we saw in Lemma \ref{epsilon-boundary_lemma}, the $\varepsilon$-boundaries of balls and their unions and intersections are fairly easy to handle.
Thus, it is convenient for us to prove that the $\varepsilon$-boundary of a dyadic cube is a subset of the union of $\varepsilon'$-boundaries of balls
for some $\varepsilon' > 0$. For this, we need a new definition.

\begin{defin}
  The \emph{approximated cube} $A_\alpha^k$ is
  \begin{eqnarray*}
    A_\alpha^k := B\left(z_\alpha^k, r_k\right) \cup \left( B\left(z_\alpha^k, R_k\right) \setminus \left( \bigcup_{\theta \neq \alpha} 
    B\left(z_\theta^k,r_k\right) \cup \bigcup_{\theta < \alpha} B\left(z_\theta^k, R_k\right) \right) \right),
  \end{eqnarray*}
  where $r_k$ is the inner radius of level $k$ and $R_k$ is the outer radius of level $k$.
\end{defin}
By inspecting the proof of Lemma \ref{partial_order_of_dyadic_points}, we notice that the approximated cubes give us an alternative way to
define the children of the pairs $(k, \alpha)$:
\begin{align*}
  (k+1,\beta) \le (k, \alpha) \ \ \ \text{ if and only if } \ \ \ z_\beta^{k+1} \in A_\alpha^k.
\end{align*}
The approximated cube is a sort of rough version of the half open dyadic cube $Q_\alpha^k$: it gives us some idea of the structure of the actual 
cube but it does not give us all the details. For example, the approximated cubes partition the space but they cannot be expressed as 
a union of smaller approximated cubes. Their structure depends highly on the indices $\alpha \in \Nc_k$.

Using the basic properties of doubling metric spaces, we can prove the following properties of approximated dyadic cubes.

\begin{lemma}
  \label{numberofiters}
  Every $x \in X$ can belong to at most $p$ $\varepsilon$-boundaries of approximated cubes, where
  \begin{align*}
    p := \left\lfloor M \left( \frac{R_k + \varepsilon}{r_k} \right)^{\log_2 M} \right\rfloor.
  \end{align*}
\end{lemma}

\begin{proof}
  If $x \in \partial_\varepsilon A_\alpha^k$, then $x \in B(z_\alpha^k, R_k + \varepsilon)$ and thus, $z_\alpha^k \in B(x, R_k + \varepsilon)$.
  Since $B(z_\alpha^k, r_k) \cap B(z_\beta^k, r_k) = \emptyset$ if $\alpha \neq \beta$, the claim follows from Lemma \ref{dms_properties}.
\end{proof}

\begin{lemma}
  \label{epsiloniter}
  For the $\varepsilon$-boundary of the approximated cube $A_\alpha^k$ we have

  \begin{align*}
    \partial_\varepsilon A_\alpha^k \subseteq \partial_\varepsilon B\left(z_\alpha^k, r_k\right) \cup \partial_\varepsilon B\left(z_\alpha^k, R_k\right) 
    \cup \bigcup_{i = 1}^{M^4} \partial_\varepsilon B\left(z_{\theta_i}^k,r_k\right) \cup \bigcup_{j = 1}^{M^4} \partial_\varepsilon B\left(z_{\theta_j}^k, R_k\right),
  \end{align*}
  for some points $z_{\theta_i}^k \in \A_k$ and $z_{\theta_j}^k \in \A_k$. In particular, the $\varepsilon$-boundary of an approximated cube is a subset of union 
  of $\varepsilon$-boundaries of at most $2 + 2M^4$ balls, $1 + M^4$ of which are inner balls and $1 + M^4$ of which are outer balls.
\end{lemma}

\begin{proof}
  Let $A_\alpha^k$ be an approximated cube. If a ball $B(z_\theta^k, r_k)$ intersects 
  the ball $B(z_\alpha^k, R_k)$, we know that $z_\theta^k \in B(z_\alpha^k, R_k + r_k)$. Since $R_k + r_k = 5 r_k$
  and the balls $B(z_\theta^k, r_k)$ and $B(z_\gamma^k, r_k)$ are disjoint for $\theta \neq \gamma$, Lemma \ref{dms_properties} 
  implies that at most $\lfloor M 5^{\log_2 M} \rfloor \le \lfloor M 8^{\log_2 M} \rfloor = M^4$ 
  balls $B(z_\theta^k, r_k)$ intersect the ball $B(z_\alpha^k, R_k)$; we label these $\theta$ as $\theta_i$, $i = 1,2,\ldots,M^4$.

  On the other hand, if a ball $B(z_\theta^k, R_k)$ intersects 
  the ball $B(z_\alpha^k, R_k)$, we know that $z_\theta^k \in B(z_\alpha^k, R_k + R_k)$. Since $R_k + R_k = 8 r_k$
  and the balls $B(z_\theta^k, r_k)$ and $B(z_\gamma^k, r_k)$ are disjoint for $\theta \neq \gamma$, Lemma \ref{dms_properties} 
  implies that at most $\lfloor M 8^{\log_2 M} \rfloor = M^4$ balls $B(z_\theta^k, R_k)$ intersect the ball $B(z_\alpha^k, R_k)$;
  we label these $\theta$ as $\theta_j$, $j = 1,2,\ldots,M^4$.

  Using these observations, Lemma \ref{epsilon-boundary_lemma} and the definition of the approximated cube, the claim follows.
\end{proof}

As we mentioned earlier, we want to show that the $\varepsilon$-boundary of a cube is a subset of $\varepsilon'$-boundaries of balls and their 
unions and intersections for some $\varepsilon' > 0$. By Lemma \ref{epsiloniter}, it suffices to show that the $\varepsilon$-boundary of a cube is a subset of 
$\varepsilon'$-boundary of an approximated cube for some $\varepsilon' > 0$. Let us prove this in the next lemma.

\begin{lemma}
  For $\varepsilon > 0$ and $k \in \Z$ we have
  \begin{flalign}
    \label{inclusion_of_epsilon-boundaries1} &\epsb \Qb_\alpha^k \subseteq \partial_{\varepsilon + 3\delta^{k+1}} A_\alpha^k \ \text{ for any } \alpha \in \Nc_k, \\
    \label{inclusion_of_epsilon-boundaries2} &\bigcup_\alpha \epsb \Qb_\alpha^k \subseteq \bigcap_{i=1}^m \bigcup_\beta \partial_{\varepsilon + 3\delta^{k+1+i}} A_\beta^{k+i} \ \text{ for every } m \in \N \cup \{ \infty \}.
  \end{flalign}
\end{lemma}

\begin{proof}
  Notice that if $x \in \Qb_\alpha^k$, then by properties \rref{union_of_closed_dcubes} and \rref{inclusion_of_dcubes} we have
  \begin{eqnarray}
    \label{distance_from_dpoint} d(x,z_\beta^{k+1}) \le 3\delta^{k+1} \text{ for some } (k+1,\beta) \le (k,\alpha)
  \end{eqnarray}
  Let us first prove the property \rref{inclusion_of_epsilon-boundaries1}. Let $z \in \epsb \Qb_\alpha^k$. Now either $z \in \Qb_\alpha^k$ or $z \in \left( \Qb_\alpha^k \right)^c$, and in either case, either $z \in A_\alpha^k$ or 
  $z \in \left( A_\alpha^k \right)^c$. Let us look at the four different cases individually.
  \begin{enumerate}
    \item[i)] Let $z \in \Qb_\alpha^k \cap A_\alpha^k$. Since $z \in \epsb \Qb_\alpha^k$, there exists $z' \in \left( \Qb_\alpha^k \right)^c$ such that $d(z,z') < \varepsilon$ 
              and $z' \in \Qb_\nu^k$ for some $\nu \neq \alpha$. By \rref{distance_from_dpoint}, we know that $d(z',z_\gamma^{k+1}) < 3\delta^{k+1}$ for some $(k+1,\gamma) \le (k,\nu)$. Thus
              $z_\gamma^{k+1} \in A_\nu^k \subseteq \left( A_\alpha^k \right)^c$ and $d(z,z_\gamma^{k+1}) < \varepsilon + 3\delta^{k+1}$. Hence, 
              $z \in \partial_{\varepsilon + 3\delta^{k+1}} A_\alpha^k$.
            
    \item[ii)] Let $z \in \Qb_\alpha^k \cap \left( A_\alpha^k \right)^c$. By \rref{distance_from_dpoint}, there exists $z_\beta^{k+1} \in \A_{k+1}$ such that $d(z,z_\beta^{k+1}) \le 3\delta^{k+1}$ 
               and $(k+1,\beta) \le (k,\alpha)$. Thus, $z_\beta^{k+1} \in A_\alpha^k$ and in particular $z \in \partial_{\varepsilon + 3\delta^{k+1}} A_\alpha^k$.
             
    \item[iii)] The proof of the case $z \in \left( \Qb_\alpha^k \right)^c \cap A_\alpha^k$ is similar to the proof of ii).
  
    \item[iv)] The proof of the case $z \in \left( \Qb_\alpha^k \right)^c \cap \left( A_\alpha^k \right)^c$ is similar to the proof of i).\\
  \end{enumerate}
  
  We can prove the property \rref{inclusion_of_epsilon-boundaries2} now easily. Since every cube can be expressed as a union of its descendant cubes, using Lemma \ref{epsilon-boundary_lemma} and 
  the property \rref{inclusion_of_epsilon-boundaries1} gives us
  \begin{eqnarray*}
    \bigcup_\alpha \epsb \Qb_\alpha^k \subseteq \bigcup_\beta \epsb \Qb_\beta^{k+i} \subseteq 
    \bigcup_\beta \partial_{\varepsilon + 3\delta^{k+1+i}} A_\beta^{k+i}
  \end{eqnarray*}
  for every $i \in \N$. The claim follows from taking the intersection over $i = 1, 2, \ldots, m$.
\end{proof}

\section{Randomizing the dyadic system}
\label{random_dyadic_systems}

The structure of the dyadic cubes in Section \ref{set-up_etc} depends on two things: the dyadic points $z_\alpha^k$ and the relation $\le$ between the pairs $(k,\alpha)$, 
$k \in \Z$, $\alpha \in \Nc_k$. 
If we chose different dyadic points or different indices $\alpha$ for them or defined the relation differently, the structure of the whole system would change. 
Randomizing the points and the relation on every different level $k \in \Z$ gives us a system that has the following property: 
the probability that a fixed point ends up near the boundary of a cube of a fixed generation $k$ is at most 
$C\varepsilon^\eta$ where $\varepsilon$ tells us about the size of the boundary region and $\eta$ is a number 
from the interval $(0,1)$ (see Theorem 2.13 in \cite{auscherhytonen}). We can sharpen this result by taking a different 
approach: we randomize only the relation. More precisely, we will do this by randomizing the inner and outer radii
that we discussed in Section \ref{set-up_etc}.

Since we want the random cubes to have the same properties as the non-random cubes, let us make sure that the assumptions of Section \ref{set-up_etc} 
hold. Let $\delta \in (0, 1/60]$ and $\A_k$ be the sets of dyadic points as in Theorem \ref{construction_of_dyadic_points}. 
We will specify the choices of the following objects later but for now, let $\Omega$ be a sample space, $a_k \colon \Omega \to \{0,1,2,\ldots \lfloor 1/\delta \rfloor\}$ a random variable for every $k \in \Z$
and $\Pro$ a probability measure such that
\begin{eqnarray}
  \label{probability_measure} \Pro(a_k = T) = \frac{1}{\left\lfloor \frac{1}{\delta} \right\rfloor}
\end{eqnarray}
for every $k \in \Z$ and $T \in \{0,1,2,\ldots,\lfloor 1/\delta \rfloor\}$. We denote
\begin{eqnarray*}
  r_k &=& r_{a_k} \ := \frac{1}{4} \left( \delta^k + a_k \delta^{k+1} \right) \ \ \text{ and} \\
  R_k &=& R_{a_k} := 4r_k
\end{eqnarray*}
for every $k \in \Z$. Then the assumptions of Lemma \ref{partial_order_of_dyadic_points} are satisfied and we can use Theorem \ref{open_and_closed_dcubes}
to construct a system of dyadic cubes
\begin{eqnarray*}
  \mathscr{D}(\omega) := \{ Q_\alpha^k(\omega) : k \in \Z, \alpha \in \Nc_k \}
\end{eqnarray*}
for every $\omega \in \Omega$ when $Q_\alpha^k(\omega)$ is the dyadic cube defined by $z_\alpha^k$, $r_k(\omega)$ and $R_k(\omega)$.
The results in Section \ref{two_different} hold for both open and closed cubes so we do not need to specify if the cube $Q_\alpha^k(\omega)$ is open or closed.

\subsection{A probabilistic lemma}

Our main intrests with random dyadic systems are related to two different sample spaces and probability measures. Although the measures are different, some claims hold for 
both of them. The next lemma is important for the later sections.

\begin{lemma}
  \label{prob1ball}
  Let $k \in \Z$ and $m > 0$. For $\varepsilon > 0$, a point $y \in X$ and a uniformly distributed random variable $a \colon \Omega \to \{0, 1, 2, \ldots, \lfloor 1 / \delta \rfloor \}$ 
  we have
  \begin{align}
    \Pro \left( x \in \epsb B \left( y, m \left( \delta^k + a\delta^{k+1} \right) \right) \right) \le \frac{2 \varepsilon + m\delta^{k+1}}{m \delta^k}.
  \end{align}
\end{lemma}

\begin{proof}
  The proof is straightforward. First, we notice that 
  \begin{eqnarray*}
    && \Pro \left( x \in \epsb B \left( y, m \left( \delta^k + a\delta^{k+1} \right) \right) \right) \\
    &\overset{\text{\rref{epsilon-boundary_ball}}}{\le} &\Pro \begin{bmatrix} d(x,y) < m \left( \delta^k + a\delta^{k+1} \right) + \varepsilon \\ d(x,y) > m \left( \delta^k + a\delta^{k+1} \right) - \varepsilon \end{bmatrix} \\
    &=& \Pro \left[ \frac{d(x,y) - m\delta^k}{m\delta^{k+1}} - \frac{\varepsilon}{m\delta^{k+1}} < a < \frac{d(x,y) - m\delta^k}{m\delta^{k+1}} + \frac{\varepsilon}{m\delta^{k+1}} \right].
  \end{eqnarray*}
  The length of this open interval is $2 \varepsilon / m\delta^{k+1}$ and thus, at most $\left\lceil 2 \varepsilon / m\delta^{k+1} \right\rceil$ integers belong to this interval. Since $a$ is 
  uniformly distributed and $a(\omega) \in \{ 0, 1, \ldots, \lfloor 1 / \delta \rfloor \}$ for every $\omega \in \Omega$, we see that
  \begin{eqnarray*}
    \Pro \left( x \in \epsb B \left( y, m \left( \delta^k + a\delta^{k+1} \right) \right) \right) &\le& \frac{\left\lceil \frac{2\varepsilon}{m\delta^{k+1}} \right\rceil}{\lfloor \frac{1}{\delta} \rfloor + 1} \\
                                                                                                    &\le& \frac{ \frac{2\varepsilon}{m\delta^{k+1}} + 1}{ \frac{1}{\delta} -1 +1 } \\
                                                                                                    & = & \frac{2\varepsilon + m\delta^{k+1}}{m\delta^k},
  \end{eqnarray*}
  which proves the claim.
\end{proof}

\begin{remark}
  \label{applying_prob1ball}
  We will apply Lemma \ref{prob1ball} for situations where $\varepsilon \le 4\delta^{k+1}$ and $m \in \{1, 1/4\}$. For these values we have
  \begin{eqnarray*}
    &\Pro \left( x \in \epsb B\left(y, \delta^k + a\delta^{k+1} \right) \right) \le 9\delta, \\
    &\Pro \left( x \in \epsb B\left(y, \frac{1}{4} \left( \delta^k + a \delta^{k+1} \right) \right) \right) \le 33\delta.
  \end{eqnarray*}

\end{remark}

\section{Two different random dyadic systems}
\label{two_different}

As the title suggests, in this section our goals are related to two different random dyadic systems. 
Our first intrest is to introduce the so called independent random dyadic systems 
and prove their property of small boundary regions with respect to the natural probability measure. After this we construct boundedly many adjacent dyadic systems 
and prove some of their properties using convenient probabilistic arguments. Both kinds of systems, with somewhat different conditions, were already known from \cite{hytonenmartikainen} and \cite{hytonenkairema}, respectively.
The present approach not only improves but also unifies these constructions, in that both systems are here viewed as special cases of the general randomization procedure introduced above. 
The proofs related to both of these systems rely on the results of Section 
\ref{boundaries_and_iterations} and Lemma \ref{prob1ball}.

\subsection{Independent random dyadic systems}
\label{independent_random_dyadic_systems}

In an \emph{independent random dyadic system} we randomize the inner and outer radii independently for every level $k \in \Z$. 
We can do this by using the sample space
\begin{eqnarray}
  \Omega := \left\{ 0, 1, 2, \ldots, \left\lfloor \frac{1}{\delta} \right\rfloor \right\}^\Z = \left\{ (\omega_k)_{k \in \Z} : \omega_k \in \left\{0,1,2,\ldots, \left\lfloor \frac{1}{\delta} \right\rfloor \right\} \right\}
\end{eqnarray}
with the natural product probability measure $\Pro_\omega$ that satisfies \rref{probability_measure} for the random variables $a_k \colon \Omega \to \{0,1,2,\ldots,\lfloor 1/\delta \rfloor\}$,
\begin{eqnarray*}
  a_k( (\omega_i)_{i \in \Z}) = \omega_k.
\end{eqnarray*}
We call the collection $\mathscr{D}_\delta := \{ \mathscr{D}(\omega) : \omega \in \Omega \}$ an independent random 
dyadic system although it is actually a collection of systems of dyadic cubes.

The following theorem shows us that our independent random dyadic systems satisfy a sharper smallness of boundary condition than 
the random systems in \cite{auscherhytonen, hytonenkairema, hytonenmartikainen}.

\begin{theorem}
  \label{smallness_of_boundary_theorem}
  Let $\varepsilon > 0$. For an independent random dyadic system with $\delta \in (0, 1/(84M^8) )$ we have
  \begin{eqnarray}
    \Pro_\omega \left( x \in \bigcup_\alpha \epsb \Qb_\alpha^k (\omega) \right) \le C_\delta \left(\frac{\varepsilon}{\delta^k}\right)^{\eta_\delta}
  \end{eqnarray}
  for constants
  \begin{eqnarray*}
    C_\delta := \frac{1}{\delta} \ \ \ \text{ and } \ \ \ \eta_\delta := 1 - \frac{\log C}{\log \left( \frac{1}{\delta} \right)},
  \end{eqnarray*}
  where $C := 84M^8.$
  In particular,
  \begin{eqnarray}
    \lim_{\delta \to 0} \eta_\delta = 1.
  \end{eqnarray}
\end{theorem}

We call the number $\eta_\delta$ the \emph{boundary exponent} of $\mathscr{D}_\delta$.\\

\begin{proof}[Proof of Theorem \ref{smallness_of_boundary_theorem}]
  First, notice that because $\delta < 1/(84M^8)$, we have $\eta_\delta \in (0,1)$. 
  If $\varepsilon > \delta^{k+1}$, then
  \begin{align*}
    \frac{1}{\delta} \left(\frac{\varepsilon}{\delta^k} \right)^{\eta_\delta} > \frac{1}{\delta} \delta^{\eta_\delta} \ge \frac{1}{\delta} \delta = 1.
  \end{align*}
  Hence, we can assume that $\varepsilon \le \delta^{k+1}$. Let $L > 0$ be the natural number such that
  \begin{align*}
    \delta^{k+L+1} < \varepsilon \le \delta^{k+L}.
  \end{align*}
  Then $\varepsilon \le \delta^{k+m}$ for every $m \le L$. Notice that $A_\beta^{k+m}(\omega) = A_\beta^{k+m}(\omega_{k+m})$ for every $m \le L$ and $\omega \in \Omega$. Thus, we see that
  \begin{align*}
    \Pro_\omega \left( x \in \bigcup_\alpha \epsb \bar{Q}_\alpha^k(\omega)\right) &\overset{\text{\rref{inclusion_of_epsilon-boundaries2}}}{\le}
    \Pro_\omega \left( x \in \bigcap_{m = 0}^{L-1} \bigcup_\beta \partial_{\varepsilon+3\delta^{k+m+1}} A_\beta^{k+m}(\omega)\right) \\
    &= \Pro_\omega \left( \left\{ \omega \in \Omega : x \in \bigcap_{m = 0}^{L-1} \bigcup_\beta \partial_{\varepsilon+3\delta^{k+m+1}} 
    A_\beta^{k+m}(\omega_{k+m}) \right\} \right) \\
    &= \Pro_\omega \left( \bigcap_{m = 0}^{L-1} \left\{ \omega \in \Omega : x \in \bigcup_\beta \partial_{\varepsilon+3\delta^{k+m+1}} 
    A_\beta^{k+m}(\omega_{k+m}) \right\} \right) \\
    &= \prod_{m=0}^{L-1} \Pro_\omega \left( \left\{ \omega \in \Omega : x \in \bigcup_\beta \partial_{\varepsilon+3\delta^{k+m+1}} 
    A_\beta^{k+m}(\omega_{k+m}) \right\} \right).
  \end{align*}
  By Lemma \ref{numberofiters}, we know that there are at most $M^4$ approximated cubes of given level that contain $x$. 
  Furthermore by Lemma \ref{epsiloniter}, we know that every $\varepsilon$-boundary of an approximated cube is a subset of a union 
  of at most $4M^4$ $\varepsilon$-boundaries of balls, $2M^4$ of which are inner balls of radius $r_{k+m}(\omega) = (1/4)(\delta^{k+m} + \omega_{k+m} \delta^{k+m+1})$ 
  and $2M^4$ of which are outer balls of radius $R_{k+m}(\omega) = \delta^{k+m} + \omega_{k+m} \delta^{k+m+1}$. 
  Since $\varepsilon+3\delta^{k+m+1} \le 4\delta^{k+1}$, we get
  \begin{eqnarray*}
    \bigcup_\beta \partial_{\varepsilon+3\delta^{k+m+1}} A_\beta^{k+m}(\omega)
    &\overset{\ref{numberofiters}}{=}& \bigcup_{j = 1}^{M^4} \partial_{\varepsilon+3\delta^{k+m+1}} A_{\beta_j}^{k+m}(\omega) \\
    &\overset{\ref{epsiloniter}}{\subseteq}& \bigcup_{i = 1}^{2M^8} \partial_{4\delta^{k+m+1}} B\left(x_i, r_{k+m}(\omega) \right)
                                        \cup \bigcup_{i = 1}^{2M^8} \partial_{4\delta^{k+m+1}} B\left(x_i, R_{k+m}(\omega) \right),
  \end{eqnarray*}
  Thus, by Remark \ref{applying_prob1ball}, we see that
  \begin{eqnarray*}
    \prod_{m=0}^{L-1} \Pro_\omega \left( \left\{ \omega \in \Omega : x \in \bigcup_\beta \partial_{\varepsilon+3\delta^{k+m+1}} 
    A_\beta^{k+m}(\omega_{k+m}) \right\} \right)
    &\le& \prod_{m=0}^{L-1} \left( \sum_{i=1}^{2M^8} 33\delta + \sum_{i=1}^{2M^8} 9\delta \right) \\
    &\le& (84M^8\delta)^L \\
    &=& (C\delta)^L.
  \end{eqnarray*}
  Since $C\delta = \delta^{\log (C\delta) / \log \delta} = \delta^{\eta_\delta}$ and $\eta_\delta \ge 0$, we see that
  \begin{align*}
    (C\delta)^L = (\delta^L)^{\eta_\delta} \le \left( \frac{\varepsilon}{\delta^{k+1}} \right)^{\eta_\delta} \le C_\delta \left( \frac{\varepsilon}{\delta^k} \right)^{\eta_\delta}
  \end{align*}
  which is what we wanted.
\end{proof}

As a simple corollary of Theorem \ref{smallness_of_boundary_theorem} we get the following result.
\begin{corollary}
  \label{smallness_of_boundary_corollary}
  For every $x \in X$ and a dyadic cube $Q_\alpha^k(\omega)$ we have
  \begin{align}
    \label{vanishboundary}
    \Pro_\omega \left( x \in \partial Q_\alpha^k(\omega) \right) = 0.
  \end{align}
\end{corollary}

The corollary holds also for the random dyadic cubes in \cite{hytonenkairema}
and \cite{auscherhytonen} (see Theorem 5.6 in \cite{hytonenkairema} and Theorem 2.13 in \cite{auscherhytonen}).

\subsection{Boundedly many adjacent dyadic systems}

In $\R^n$ it is very easy to construct a finite number of adjacent dyadic systems without adding a random element to the systems: 
one example of these systems is
\begin{align}
  \label{adjacent_systems_Rn}
  \mathscr{D}(t) := \left\{ 2^{-k} ([0,1)^n + m + (-1)^k t) : k \in \Z, m \in \Z^n \right\}, \ \ \ t \in \{0, 1/3, 2/3\}^n,
\end{align}
as in \cite{hytonenlaceyperez}. In general doubling metric spaces it is perhaps easier to think of the similar systems as a special case of the random dyadic systems
since this makes it possible to use probabilistic arguments in the proofs. 

With an independent random dyadic systems we used an infinite sample space and thus we had an infinite number of different dyadic systems. 
By taking a finite sample space
\begin{align}\label{finiteOmega}
  \Omega := \left\{ 0, 1, 2, \ldots, \left\lfloor \frac{1}{\delta} \right\rfloor \right\},
\end{align}
the natural probability measure on a finite set and the random variable $a \colon \Omega \to \{0,1,2,\ldots,\lfloor 1/\delta \rfloor\}$, $a(\omega) = \omega$,
we get only a finite number of systems but at the same time we give up the independence we used previously. In other words, every 
$\omega \in \Omega$ defines all the radii of all the cubes of every level in the same way. Namely: we 
have $r_k(\omega) := (1/4) \left(\delta^k + \omega\delta^{k+1}\right)$ and $R_k(\omega) := 4r_k(\omega)$ for every $k \in \Z$.

The benefit of adjacent dyadic systems is to provide an efficient tool for approximating geometric objects (like arbitrary balls, which are uncountable in number and often require subtle covering lemmas to deal with) by dyadic ones (like dyadic cubes, which are countable in number and have essentially trivial covering properties). Here we show that the adjacent dyadic systems just described have the following approximation property, which strengthens the earlier versions in abstract spaces:

\begin{theorem}
  \label{adjacent_systems_theorem}
  Let $X$ be a doubling metric space and $\delta < 1 / (168M^8)$, where $M$ is the doubling constant of $X$. Let $\mathscr{D}(\omega)$, $\omega\in\Omega$, be the adjacent dyadic systems as defined after \eqref{finiteOmega}.
  Then for every ball $B := B(x,r) \subseteq X$ and every $p \in \N$, there is an $\omega \in \Omega$ and a cube $Q \in \mathscr{D}(\omega)$ such that
  \begin{eqnarray}
    \label{ball_within_cube}               &B \subseteq Q, \\
    \label{sidelength_of_a_cube}           &l(Q) \le \delta^{-2} r(B), \\
    \label{enlargement_within_enlargement} &\left(\frac{1}{\delta}\right)^p B \subseteq Q^{(p)},
  \end{eqnarray}
  where $l(Q) = \delta^k$ if $Q = Q_\alpha^k$, $r(B)$ is the radius of $B$ and $Q^{(p)}$ is the unique ancestor of $Q$ of generation $k-p$.
\end{theorem}

We make a few remarks on the history and applications of this kind of results. Without the last condition \rref{enlargement_within_enlargement}, the analogous theorem is well known in $\R^n$ for the system \rref{adjacent_systems_Rn}. The extension to metric spaces (still without \rref{enlargement_within_enlargement}) was obtained in \cite{hytonenkairema} and applied to weighted norm inequalities in \cite{hytonenkairema,kairema}. In $\R^n$, the analogue of Theorem~\ref{adjacent_systems_theorem} with all conditions was established in \cite[Lemma 2.5]{hytonenlaceyperez}, again with applications to weighted norm inequalities in the same paper. Motivated by its usefulness, and with potential metric space extensions of \cite{hytonenlaceyperez} in mind,  we here extend this stronger version to doubling metric spaces, noting that the previous non-probabilistic construction in \cite{hytonenkairema} does not seem to allow for such a strengthening.  Our new construction allows us to use the results of Section 
\ref{boundaries_and_iterations} and Lemma \ref{prob1ball}, which help us to give a fairly straightforward proof for the theorem.

\begin{proof}[Proof of Theorem \ref{adjacent_systems_theorem}]
  Let $B(x,r) \subseteq X$. Choose $k \in \Z$ such that $\delta^{k+2} < r \le \delta^{k+1}$. By Lemma \ref{numberofiters}, there are at most 
  $M^4$ indices $\alpha$ such that $x \in \partial_{\delta^{k+1}} A_\alpha^k$ and by Lemma \ref{epsiloniter}, we know that the $\varepsilon$-boundary 
  of an approximated cube is a subset of at most $4M^4$ $\varepsilon$-boundaries of balls. $2M^4$ of these balls are inner balls of radius 
  $r_k(\omega) = (1/4) \left( \delta^k + \omega\delta^{k+1}\right)$ and $2M^4$ of these balls are outer balls of radius $R_k(\omega) = 4r_k(\omega)$. Thus, we get
  \begin{eqnarray*}
    \bigcup_\alpha \partial_{\delta^{k+1}} Q_\alpha^k(\omega) &\overset{\text{\rref{inclusion_of_epsilon-boundaries1}}}{\subseteq}&
    \bigcup_\alpha \partial_{4\delta^{k+1}} A_\alpha^k(\omega) \\
    &\overset{\ref{numberofiters}}{\overset{\ref{epsiloniter}}{\subseteq}}&
    \bigcup_{i = 1}^{2M^8} \partial_{4\delta^{k+1}} B\left(x_i, r_k(\omega) \right)
                                        \cup \bigcup_{i = 1}^{2M^8} \partial_{4\delta^{k+1}} B\left(x_i, R_k(\omega) \right),
  \end{eqnarray*}
  and furthermore
  \begin{align*}
    \Pro_\omega \left( x \in \bigcup_\alpha \partial_{\delta^{k+1}} Q_\alpha^k(\omega) \right) 
    \overset{\ref{applying_prob1ball}}{\le} 2M^8 \left( 33\delta + 9\delta \right) = 84M^8 \delta
  \end{align*}
  Similarly,
  \begin{align*}
    \Pro_\omega \left( x \in \bigcup_\alpha \partial_{\delta^{k-p+1}} Q_\alpha^{k-p}(\omega) \right) \le 84M^8\delta
  \end{align*}
  and therefore
  \begin{align*}
    \Pro_\omega \left( x \in \left( \bigcup_\alpha \partial_{\delta^{k-p+1}} Q_\alpha^{k-p}(\omega) \cup \bigcup_\alpha 
    \partial_{\delta^{k+1}} Q_\alpha^k(\omega) \right) \right) \le 168M^8 \delta < 1.
  \end{align*}
  Thus the complement event has probability $1 - 168M^8 \delta > 0$. Hence, there exists an $\omega \in \Omega$ such that
  \begin{align}
    \label{x_not_in}
    x \notin \bigcup_\alpha \partial_{\delta^{k-p+1}} Q_\alpha^{k-p}(\omega) \cup \bigcup_\alpha 
    \partial_{\delta^{k+1}} Q_\alpha^k(\omega)
  \end{align}

  Let $Q_\alpha^k(\omega) \ni x$. Now \rref{x_not_in} implies that
  \begin{align*}
    d\left( x, \left(Q_\alpha^k(\omega)\right)^c\right) \ge \delta^{k+1} \ge r
  \end{align*}
  and hence, $B(x,r) \subseteq Q_\alpha^k (\omega)$. Because now $x \in Q_\alpha^k(\omega)^{(p)} =: Q_\theta^{k-p}(\omega)$, 
  \rref{x_not_in} also implies that
  \begin{align*}
    d\left( x, \left(Q_\alpha^{k-p}(\omega)\right)^c\right) \ge \delta^{k-p+1} \ge \delta^{-p}r,
  \end{align*}
  and hence, $B(x,\delta^{-p}r) \subseteq Q_\theta^{k-p}(\omega)$. We also see that
  \begin{align*}
    l(Q_\alpha^k) = \delta^k = \delta^{-2} \delta^{k+2} < \delta^{-2} r,
  \end{align*}
  which completes the proof.
\end{proof}

\section{Application: H\"older-continuous splines and wavelets}
\label{application}

In this section we indicate the consequences of our new random cubes for the existence of H\"older-continuous spline and wavelet bases in abstract metric measure spaces. As mentioned before, we can essentially just feed our new cubes into the construction of \cite{auscherhytonen} and collect the results. For concreteness, let us nevertheless recall the relevant definitions.

Besides being of independent interest, the spline functions serve as a natural intermediate step between the random cubes and the wavelets. Thus we consider them first.
Let $(X,d)$ be a doubling metric space.

\begin{defin}
  A set of functions $s_\alpha^k: X \to [0,1]$ is a \emph{system of spline functions} if the following properties are satisfied for some points $\mathscr{Z}^k:=\{z^k_\alpha\}_{\alpha}$, where $\mathscr{Z}^k\subseteq\mathscr{Z}^{k+1}\subseteq X$ for all $k$, and for constants $\delta \in (0,1)$ and $C>c> 0$:
  \begin{eqnarray}
    \text{bounded support: } && 1_{B(z_\alpha^k, c\delta^k)} (x) \le s_\alpha^k(x) \le 1_{B(z_\alpha^k, C\delta^k)}(x), \label{boundedsupp} \\ %
    \text{interpolation: }   && s_\alpha^k(z_\beta^k) = \begin{cases} 1, & \mbox{if } \alpha = \beta \\ 0, & \mbox{if } \alpha \neq \beta \end{cases}, \label{interpolation} \\
    \text{partition: }    && \sum_\alpha s_\alpha^k(x) = 1, \label{splinesum} \\
    \text{refinement: }   && s_\alpha^k(x) = \sum_\beta p_{\alpha\beta} \cdot s_\beta^{k+1}(x), \label{refinement}
  \end{eqnarray}
  where $\{p_{\alpha\beta}\}_\beta$ is a finitely nonzero set of nonnegative coefficients with $\sum_\alpha p_{\alpha\beta}^k = 1$. The indices $k$ and $\alpha$ run respectively over $\Z$ and $\N$ if $X$ is unbounded, or over $\{k\in\Z:k\geq k_0\}$ and $\{0,1,\ldots,n_k\}$ for some finite $k_0\in\Z$ and $n_k\in\N$ if $X$ is bounded.

We call the splines \emph{H\"older-continuous of exponent $\eta$} if
  \begin{align*}
    | s_\alpha^k(x) - s_\alpha^k(y) | \le C \left( \frac{d(x,y)}{\delta^k} \right)^{\eta}.
  \end{align*}
If $\eta = 1$, we call the splines \emph{Lipschitz-continuous}.
\end{defin}

Although the properties \rref{boundedsupp} through \rref{refinement} are already non-trivial, it is not difficult to provide simple systems of functions that satisfy them. For example, we can just take a system of non-random half-open dyadic cubes $\mathscr{D} := \{ Q_\alpha^k : k \in \Z, \alpha \in \N \}$ (Theorem 2.2 in \cite{hytonenkairema}) and set $s_\alpha^k = 1_{Q_\alpha^k}$.

Also observe that the refinement  \eqref{refinement} is the only property that ties the splines of different generations $k$ together. In the absence of this property, it is not difficult to construct systems of function with all other properties, even Lipschitz-continuous. Such systems are well known as partitions of unity. The point of spline functions is to make a sequence of partitions of unity compatible with each other in the sense of the refinement property \eqref{refinement}.

In the abstract set-up, this issue was first addressed in \cite{auscherhytonen}, where the following connection with random dyadic cubes was established: (The result is not explicitly formulated in this way in \cite{auscherhytonen}, but it can be easily read from the proof of Theorem 3.1 in \cite{auscherhytonen}.)

\begin{theorem}[\cite{auscherhytonen}]\label{regular_splines_theorem}
If a doubling metric space supports an  independent random dyadic system $\mathscr{D}$ of boundary exponent $\eta$, then it also supports a system of H\"older-continuous splines with the same exponent $\eta$. In fact, such a system can be defined by
\begin{equation}\label{eq:splineFormula}
  s_\alpha^k(x) := \Pro_\omega \left( x \in \Qb_\alpha^k(\omega) \right).
\end{equation}
\end{theorem}

In \cite{auscherhytonen}, this gave the existence of H\"older-continuous splines with some small exponent $\eta>0$; combined with our new Theorem \ref{smallness_of_boundary_theorem}, it gives the following:

\begin{corollary}
  In every doubling metric space there exists a system of H\"older-continuous spline functions of every exponent $\eta \in [0,1)$.
\end{corollary}

We then proceed to wavelets. The setting is now a doubling metric space $(X,d)$ equipped with a Borel measure $\mu$ with the doubling property
\begin{equation*}
  \mu(B(x,2r))\leq C\mu(B(x,r)).
\end{equation*}

\begin{defin}
A set of functions $\psi^k_\alpha:X\to\R$ is a \emph{basis of wavelets with $\varrho$-localization},  where $\varrho:[0,\infty)\to[0,1]$ is a non-increasing function, if the following properties are satisfied for some points $\mathscr{Y}^k:=\{y^k_\alpha\}_\alpha\in X$ and constants $\delta \in (0,1)$ and $C > 0$:
  \begin{eqnarray}
    \text{vanishing mean:} &&\int\psi^k_\alpha(x)d\mu(x)=0, \\
    \text{localization:} && \abs{\psi^k_\alpha(x)}\leq\frac{C}{\sqrt{\mu(B(y^k_\alpha,\delta^k))}}\varrho\Big(\frac{d(x,y^k_\alpha)}{\delta^k}\Big), 
  \end{eqnarray}
and the functions $\psi^k_\alpha$ form an orthonormal basis of $L^2_0(\mu)$, where
\begin{equation*}
  L^2_0(\mu):=\begin{cases} L^2(\mu), & \text{if $X$ is unbounded}, \\
     \Big\{f\in L^2(\mu):\int_X f(x)d\mu(x)=0\Big\}, & \text{if $X$ is bounded}.\end{cases}
\end{equation*}
The indices $k$ and $\alpha$ run over similar sets as in the case of splines.
% respectively over $\Z$ and $\N$ if $X$ is unbounded, or over $\{k\in\Z:k\geq k_0\}$ and $\{0,1,\ldots,n_k\}$ for some finite $k_0\in\Z$ and $n_k\in\N$ if $X$ is bounded.

We call the $\varrho$-localized wavelets \emph{H\"older-continuous of exponent $\eta$} if
  \begin{align*}
    | \psi_\alpha^k(x) - \psi_\alpha^k(y) | \leq \frac{C}{\sqrt{\mu(B(y^k_\alpha,\delta^k))}}\varrho\Big(\frac{d(x,y^k_\alpha)}{\delta^k}\Big)
         \left( \frac{d(x,y)}{\delta^k} \right)^{\eta}
  \end{align*}
If $\eta = 1$, we call the wavelets \emph{Lipschitz-continuous}.
\end{defin}

Let us notice that on $\R^d$ or other symmetric spaces, one usually imposes additional self-similarity properties on the wavelets. However, these are hardly meaningful in the generality that we consider, so we insist on this reduced definition.

Two main cases of $\varrho$-localization that we have in mind are:
\begin{itemize}
  \item \emph{perfect localization}: $\varrho=1_{[0,c]}$ for some finite $c\in(0,\infty)$.
  \item \emph{exponential localization}: $\varrho(x)=\exp(-\gamma x)$ for some $\gamma>0$.
\end{itemize}

Without the H\"older-continuity requirement, perfect localization is achieved by the Haar functions, which are readily constructed from the indicators of (non-random) dyadic cubes in the generality of abstract spaces (see e.g.~\cite{hytonenIMRN}).

The existence of H\"older-continuous wavelets with perfect localization (akin to the celebrated Daubechies wavelets on $\R^d$) remains an interesting open problem in abstract metric spaces. For exponential localization, the following connection with the spline bases was established in \cite{auscherhytonen}: (Once again, this explicit statement is not found in \cite{auscherhytonen}, but it can be readily read from the considerations in \cite{auscherhytonen}, Sections 5 and 6.)

\begin{theorem}[\cite{auscherhytonen}]\label{thm:splinesImpliesWavelets}
If a doubling metric space supports a system of H\"older-continuous splines with exponent $\eta$, then it also supports a H\"older-continuous wavelet basis of exponential localization, with the same H\"older-exponent $\eta$.
\end{theorem}

In fact, the construction essentially adapts a classical algorithm from \cite{meyer}, but it is somewhat more complicated than the simple formula \eqref{eq:splineFormula}, so we refer to \cite{auscherhytonen} for details. Let us only point out that the wavelets constructed in this way will be localized around points $y^k_\alpha$, where
\begin{equation*}
  \mathscr{Y}^k:=\{y^k_\alpha\}_\alpha=\mathscr{Z}^{k+1}\setminus\mathscr{Z}^k,
\end{equation*}
and $\mathscr{Z}^k:=\{z^k_\alpha\}_\alpha$ is the point set related to the corresponding splines.

 In \cite{auscherhytonen}, Theorem~\ref{thm:splinesImpliesWavelets} gave the existence of H\"older-continuous wavelets with some small exponent $\eta>0$; combined with our new Theorem~\ref{smallness_of_boundary_theorem} (and Theorem~\ref{regular_splines_theorem}) it gives:

\begin{corollary}
  In every doubling metric space there exists a basis of H\"older-continuous wavelets with exponential localization, for any H\"older-exponent $\eta \in [0,1)$.
\end{corollary}

We conclude by summarizing some of the related open problems:

\begin{openproblem}
Do the following systems of functions exist in every doubling metric (measure) space:
\begin{enumerate}
  \item\label{it:cubes} a system of independent random dyadic cubes with boundary exponent one?
  \item\label{it:splines} a system of Lipschitz-continuous splines?
  \item\label{it:wavelets} a basis of Lipschitz-continuous wavelets with exponential localization?
  \item\label{it:perfect} a basis of H\"older-continuous wavelets with perfect localization?
\end{enumerate}
\end{openproblem}

By Theorem~\ref{regular_splines_theorem}, an affirmative answer to \eqref{it:cubes} would imply an affirmative answer to \eqref{it:splines}, which would in turn imply an affirmative answer to \eqref{it:wavelets} by Theorem~\ref{thm:splinesImpliesWavelets}, but potentially there could be other approaches to these problems. The last question~\eqref{it:perfect} appears to be disjoint from these direct chains of implications, but it is nevertheless recorded due to its natural proximity.

\appendix

\section{Construction of dyadic points}

In this appendix, we give a proof of Theorem  \ref{construction_of_dyadic_points} on the existence of systems of dyadic points in an abstract metric space.

For clarity, let us define couple of different types of sets:
\begin{defin*}
  Let $S$ be a set and $E$ be another set, which may or may not contain $S$.
  \begin{enumerate}
    \item[i)] The set $S$ is \emph{$r$-separated} if any two distinct points $x,y \in S$ satisfy $d(x,y) \ge r$.
    \item[ii)] The set $S$ is \emph{maximal $r$-separated within $E$} if $S\cup\{z\}$ is not $r$-separated for any $z \in E \setminus S$.
    \item[iii)] A set $R$ is \emph{$r$-separated extension of $S$ within $E$} if $S \subseteq R \subseteq S \cup E$ and $R$ is also $r$-separated. 
  \end{enumerate}
\end{defin*}
For example, the set $\{1\}$ is a maximal $1$-separated within the interval $(0,1)$. It is an immediate consequence of Zorn's lemma that if $S$ is $r$-separated and $E$ is another set, then there exists a maximal $r$-separated extension of $S$ within $E$.\\

\begin{proof}[Proof of Theorem \ref{construction_of_dyadic_points}]
  We construct sets $\mathcal{C}^n_k$, for integers $0\leq k<n<\infty$, with the following properties:
  \begin{enumerate}
    \item\label{it:monot} $\mathcal{C}^n_k$ is increasing in $n$ and decreasing in $k$;
    \item\label{it:sep} $\mathcal{C}^n_k$ is $\Delta^k$-separated;
    \item\label{it:maximal} $\mathcal{C}^n_k$ is maximal $\Delta^k$-separated within $B(x_0,R^n_k)$, where $R^n_k:=\Delta^n-\sum_{i=k}^{n-1}\Delta^i$;
    \item\label{it:region} $\mathcal{C}^n_k\subseteq B(x_0,R^n_k)\cup\mathcal{C}^n_{k+1}$ if $k+1<n$ and $\mathcal{C}^n_{n-1}\subseteq B(x_0,R^n_{n-1})$;
    \item\label{it:unionSep} $\mathcal{C}^n_i\cup\mathcal{C}^{n+1}_k$ is  $\Delta^i$-separated for all $i\leq k$.
  \end{enumerate}

  The construction will proceed recursively along the following ordering of the pairs $(n,k)$:
  \begin{equation*}
  \begin{split}
    &(1,0)\prec(2,1)\prec(2,0)\prec(3,2)\prec(3,1)\prec(3,0)\prec\ldots \\
    &\qquad \prec(m,0)\prec(m+1,m)\prec\ldots\prec(m+1,j+1)\prec(m+1,j)\prec\ldots
  \end{split}
  \end{equation*}
  In the initial step, let $\mathcal{C}^1_0\subseteq B(x_0,\Delta-1)$ be a maximal one-separated set that contains $x_0$.

  In the inductive step, we assume that $\mathcal{C}^n_k$ has already been constructed for all $(n,k)\prec(m+1,j)$, in such a way that all above listed properties \eqref{it:monot} through \eqref{it:unionSep} hold, whenever the relevant indices are smaller than $(m+1,j)$ with respect to $\prec$.  Our task is to construct $\mathcal{C}^{m+1}_j$ in such a way that these properties stay valid whenever the relevant indices are smaller than or equal to $(m+1,j)$.

  \begin{figure}[h]
    \begin{center}
      \setlength{\fboxsep}{0pt}
      \setlength{\fboxrule}{1pt}
      \fbox{\includegraphics[scale=0.27]{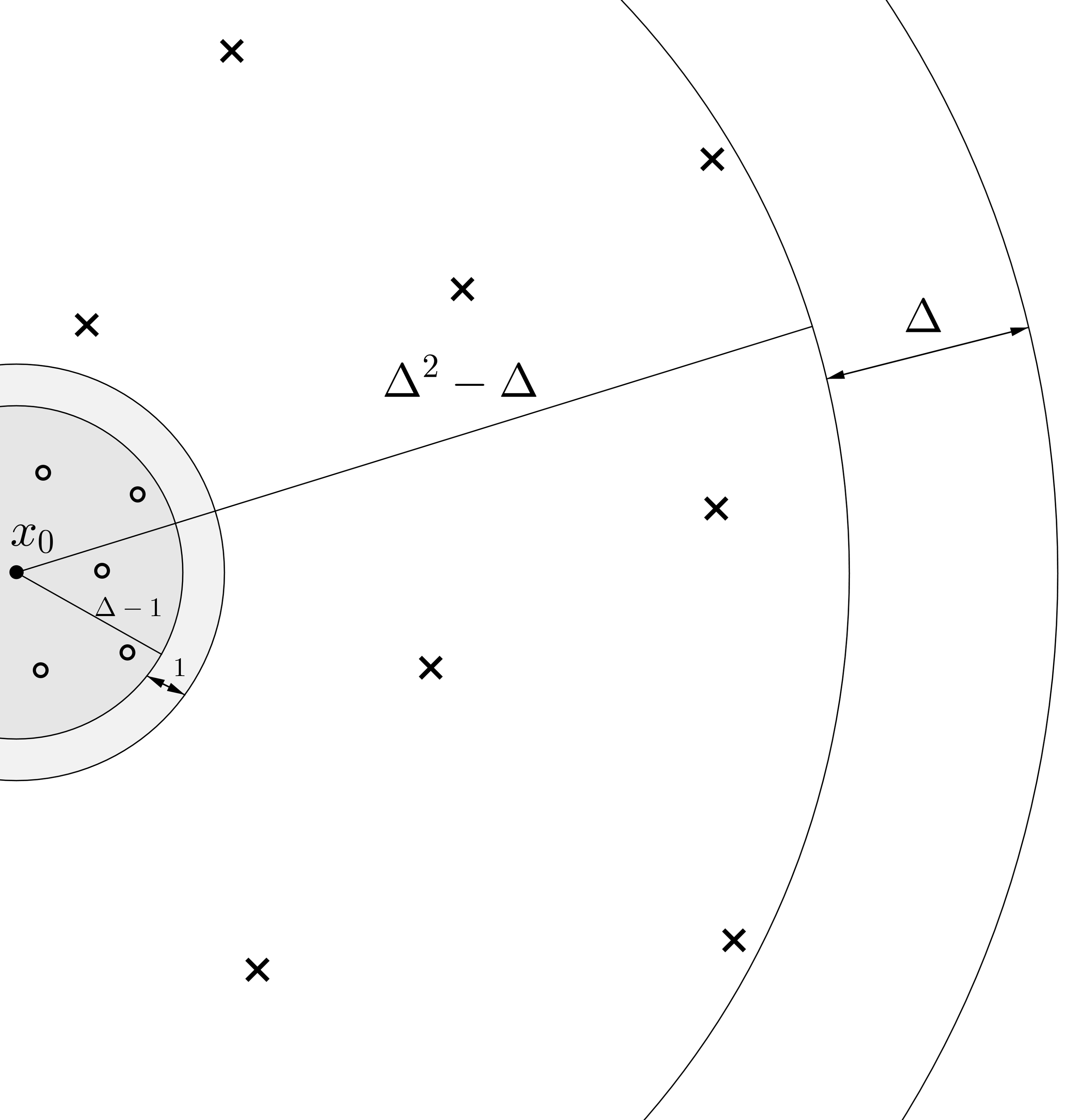}} \fbox{\includegraphics[scale=0.27]{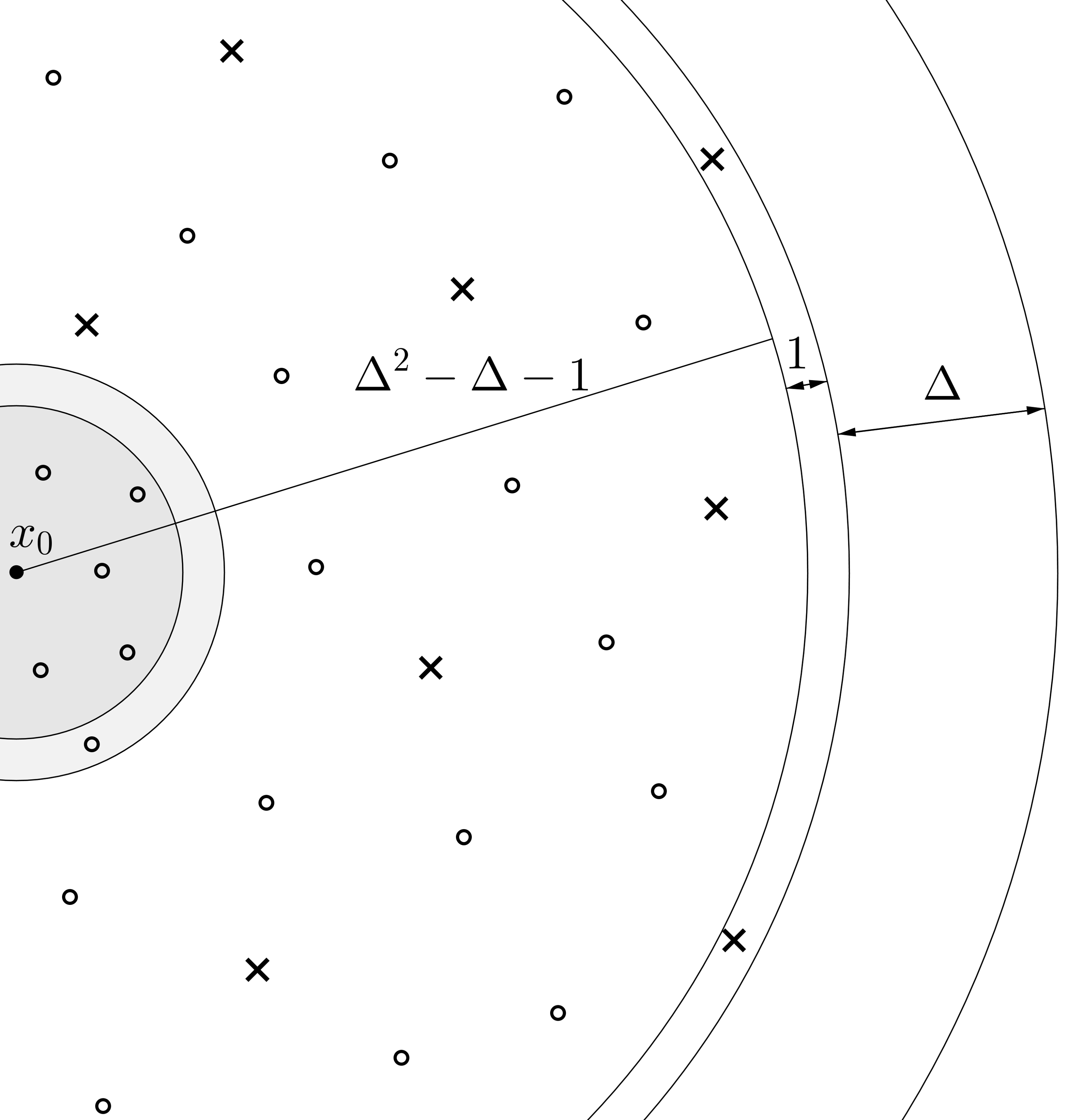}}
      \caption{The first three steps of the construction. Left: The points $\mathcal{C}^1_0$ (shown as~$\circ$) in the innermost disk $B(x_0,\Delta-1)$, and the points $\mathcal{C}^2_1$ (shown as~$\times$) in the annulus $B(x_0,\Delta^2-\Delta)\setminus B(x_0,\Delta)$. Right: The points $\mathcal{C}^2_0$, which includes $\mathcal{C}^1_0\cup\mathcal{C}^1$ as well as new points (shown as~$\circ$) in the annulus $B(x_0,\Delta^2-\Delta-1)\setminus B(x_0,\Delta-1)$.} 
    \end{center}
  \end{figure}

  \subsubsection*{Case $j=m$}
  Let $\mathcal{C}^{m+1}_m\subseteq B(x_0,R^{m+1}_m)=B(x_0,\Delta^{m+1}-\Delta^m)$ be a maximal $\Delta^m$-separated set that contains $x_0$. Thus $\mathcal{C}^{m+1}_m=\{x_0\}\cup\mathcal{N}^{m+1}_m$ with $\mathcal{N}^{m+1}_m\subseteq B(x_0,\Delta^m)^c$.

  From \eqref{it:region} it follows that $\mathcal{C}^m_k\subseteq B(x_0,R^m_{m-1})=B(x_0,\Delta^m-\Delta^{m-1})$ for all $k\leq m-1$. Then it is clear that $\mathcal{C}^m_k\cup\mathcal{C}^{m+1}_m$ is $\Delta^k$-separated. The other properties are immediate to check.

  \subsubsection*{Case $j<m$}
  Let $\mathcal{C}^{m+1}_j$ be a maximal $\Delta^j$-separated extension of $\mathcal{C}^m_j\cup\mathcal{C}^{m+1}_{j+1}$ (which is itself $\Delta^j$-separated by \eqref{it:unionSep}) within $B(x_0,R^{m+1}_j)$. Thus $\mathcal{C}^{m+1}_j=\mathcal{C}^m_j\cup\mathcal{C}^{m+1}_{j+1}\cup\mathcal{N}^{m+1}_j$, where $\mathcal{N}^{m+1}_j\subseteq B(x_0,R^{m+1}_j)\setminus B(x_0,R^m_j)$. (We know that $\mathcal{N}^{m+1}_j\subseteq B(x_0,R^m_j)^c$, since $\mathcal{C}^m_j$ is already maximal $\Delta^j$-separated within $B(x_0,R^m_j)$ by \eqref{it:maximal}.)

  Most of the properties are straightforward to verify, and we concentrate on \eqref{it:unionSep} for $(n+1,k)=(m+1,j)$. We proceed by backwards induction of $i\leq j$.

  If $i=j$, then $\mathcal{C}^m_j\subseteq\mathcal{C}^{m+1}_j$, so the union is just $\mathcal{C}^{m+1}_j$, which is $\Delta^j$-separated by construction.

  Let then $i<j$, and suppose that the $\Delta^{i+1}$-separation of $\mathcal{C}^m_{i+1}\cup\mathcal{C}^{m+1}_j$ has already been verified. We need to check that $\mathcal{C}^m_i\cup\mathcal{C}^{m+1}_j$ is $\Delta^i$-separated, and we know this for both sets individually, so we need to check that $d(x,y)\geq\Delta^i$ for any $x\neq y$ such that $(x,y)\in\mathcal{C}^m_i\times\mathcal{C}^{m+1}_j$.
  For contradiction, assume that $d(x,y) < \Delta^i$. Since $\mathcal{C}_{i+1}^m \cup \mathcal{C}_j^{m+1}$ is $\Delta^i$-separated by the induction assumption, from \eqref{it:region} it follows that $x \in B(x_0,R_i^m)$. Furthermore, $y \in B(x_0,R_i^m + \Delta^i) = B(x_0,R_{i+1}^m)$. Since 
  $y \in \mathcal{C}_j^{m+1} \subseteq \mathcal{C}_{j-1}^{m+1} \subseteq \ldots \subseteq \mathcal{C}_{i+1}^{m+1}$ and the set $\mathcal{C}_{i+1}^m$ is maximal $\Delta^{i+1}$-separated within $B(x_0,R_{i+1}^m)$, it follows that 
  $y \in \mathcal{C}_{i+1}^m$. Then \eqref{it:monot} implies that $y \in \mathcal{C}_i^m$, which is a contradiction since $\mathcal{C}_i^m$ is $\Delta^i$-separated by \eqref{it:sep}. Hence, $\mathcal{C}_i^m \cup \mathcal{C}^{m+1}$ is $\Delta^i$-separated.\\

  It is easy to verify that $R_k^n \to \infty$ as $n \to \infty$ for every $k \in \N$ if and only if $\Delta > 2$. Thus, for a fixed $\delta \in (0,1/2)$ we can set $\Delta = 1/\delta$ and $\A_k = \bigcup_{n=1}^\infty \mathcal{C}_{-k}^n$. After this, the existence of the sets $\A_k$ for $k > 0$ follows simply from Lemma \ref{maximal_subsets} (if the space $(X,d)$ is doubling) or Zorn's lemma (if the space $(X,d)$ is not doubling).
\end{proof}

\begin{remark*}
  Choosing the set $\A_{k-1}$ after we have chosen the whole set $\A_k$ might not be possible since such set $\A_{k-1}$ might not exist. We can see this by a simple 
  example. Let $\delta = 1/3$, $X = B(0,3) \cup B(8,3) \subseteq \R$ and $\A_{-1} = \{0,8\}$. Now 
  $0 \in B(8,(1/3)^{-2})$ and $8 \in B(0,(1/3)^{-2})$ but $X \setminus B(0,(1/3)^{-2}) \neq \emptyset \neq X \setminus B(8,(1/3)^{-2})$, so there does not exist a set $\A_{-2} \subseteq \A_{-1}$ such that 
  the points of $\A_{-2}$ are $\delta^{-2}$-separated and $\min_{z \in \A_2} d(x,z) < \delta^{-2}$ for every $x \in X$.
\end{remark*}

\bibliography{dyadic}
\bibliographystyle{plain}

\end{document}